\documentclass[10pt,reqno]{amsart}
\usepackage{float}
\usepackage{graphicx}
\usepackage{amsmath,amssymb}
\usepackage{hyperref}
 \usepackage{amsaddr}
\newtheorem{thm}{Theorem}[section]
\newtheorem{cor}[thm]{Corollary}
\newtheorem{lem}[thm]{Lemma}

\theoremstyle{definition}

\theoremstyle{remark}
\newtheorem{rem}[thm]{Remark}
\numberwithin{equation}{section}
\newcommand{\norm}[1]{\left\Vert#1\right\Vert}
\newcommand{\abs}[1]{\left\vert#1\right\vert}

\renewcommand\Re{\operatorname{Re}}
\renewcommand\Im{\operatorname{Im}}
\usepackage{scalerel}
\usepackage{stackengine}
\stackMath
\def\hatgap{2pt}
\def\subdown{-2pt}
\newcommand\reallywidehat[2][]{%
\renewcommand\stackalignment{l}%
\stackon[\hatgap]{#2}{%
\stretchto{%
    \scalerel*[\widthof{$#2$}]{\kern-.6pt\bigwedge\kern-.6pt}%
    {\rule[-\textheight/2]{1ex}{\textheight}}
}{0.5ex}
_{\smash{\belowbaseline[\subdown]{\scriptstyle#1}}}%
}}
\begin{document}

\title[Fourth order Schrödinger equation]{Dispersion estimates for the boundary integral operator associated with the fourth order Schrödinger equation posed on the half line}%
\author{T. Özsarı$^{\MakeLowercase{a},*}$, K. Alkan$^{\MakeLowercase{b}}$,  and  K. Kalimeris$^{\MakeLowercase{c}}$}
\address{$^a$Department of Mathematics, Bilkent University\\ Çankaya, Ankara, 06800 Turkey}
\address{$^b$Department of Mathematics, Izmir Institute of Technology\\ Urla, Izmir, 35430 Turkey}
\address{$^c$Academy of Athens, Mathematics Research Center, Greece}%

\thanks{*E-mail: turker.ozsari@bilkent.edu.tr}
%
%
\thanks{KA and T{\" O}'s research are funded by TÜBİTAK's 1001 Grant 117F449. TÖ's work was also partially supported by the Science Academy's Young Scientist Award Program (BAGEP 2020). KK was partially supported by the research programme (200/984) of the Research Committee of the Academy of Athens. This work was completed while TÖ was visiting the Academy of Athens in September of 2021.  }%

\begin{abstract}
    In this paper, we prove dispersion estimates for the boundary integral operator associated with the fourth order Schrödinger equation posed on the half line. Proofs of such estimates for domains with boundaries are rare and generally require highly technical approaches, as opposed to our simple treatment which is based on constructing a boundary integral operator of oscillatory nature via the Fokas method. Our method is uniform and can be extended to other higher order partial differential equations where the main equation possibly involves more than one spatial derivatives.
\end{abstract}
\maketitle
\tableofcontents
\section{Introduction}
A boundary integral operator (\texttt{BdIntOp}) associated with an initial-boundary value problem (\texttt{IBVP}) is a mapping in the form of an integral formula that takes functions defined on the boundary of the space-time cylinder to solutions of the given \texttt{IBVP}, say with zero initial datum and interior source. Regularity analysis of such operators plays a crucial role in establishing local wellposedness for (nonlinear) \texttt{IBVP}s.  A \texttt{BdIntOp} can be written in abstract or explicit form.  An abstract formula is generally based on the semigroup theory.  However, in order to place an \texttt{IBVP} within the context of semigroup theory, one needs to somehow homogenize the given boundary condition so that the domain of the generator becomes a time independent linear space.    This is generally done by first extending the given boundary input as a solution of a relevant stationary problem and then subtracting it from the original problem.  From the regularity point of view, this approach costs loss of derivatives in wellposedness analysis, and one needs to employ rather advanced techniques to retrieve desired smoothness properties.  On the other hand, there are methods to obtain explicit formulas for  \texttt{BdIntOp}s directly without using an extension-homogenization approach.  One of the most effective choices of such direct methods is the unified transform method (\texttt{UTM}), also known as the Fokas method, see e.g., \cite{F97} and \cite{F08}. This method was recently used to construct \texttt{BdIntOp} for establishing local wellposedness of nonlinear initial boundary value problems, see for instance \cite{F17}, \cite{Him20} and \cite{OY19}.  This method is realised in three main steps: (i) the construction  of a \emph{global relation}, which is an identity that relates some particular integral transforms of known and unknown boundary values and the sought after solution, (ii) the derivation of an \textit{integral representation} of the solution which involves the integral transforms of both the known and the unknown boundary values, (iii) the evaluation of the contribution of the unknown values in the integral representation, with the utilisation of the global relation. This last step requires (a) at the level of the global relation, the identification of the \emph{invariance maps}  which keep spectral inputs of the transforms of boundary values unchanged, and (b) at the level of the integral representation, a subtle \emph{contour deformation}  based on delicate complex analytic arguments.  The space-time structure of \texttt{BdIntOp}s constructed via the \texttt{UTM} allows one to use the tools of Fourier and harmonic analysis, in particular the theory of oscillatory integrals, for proving Strichartz type estimates.  These estimates are essential for establishing the low regularity theory in function spaces.

This paper aims to (i) construct a \texttt{BdIntOp} corresponding to the fourth order Schrödinger equation subject to Dirichlet-Neumann boundary conditions via the \texttt{UTM} and (ii) prove dispersion estimates (that imply Strichartz type estimates) for this \texttt{BdIntOp} with respect to boundary data. More precisely, we consider the following partial differential equation (PDE):
\begin{align}
&y_{t}+Py=0, \quad (x,t)\in\mathbb{R}_+\times (0,T),\label{maineq}\\
&y(x,0)=0,\label{init}\\
&B_jy(0,t)=g_{j}(t)\label{bdry},
\end{align} where $P$, $B_j$, $j=0,1$ are (differential) operators given by $P=-i(\partial_x^4+\partial_x^2)$, $B_0=\gamma_0$ (Dirichlet trace operator), and $B_1=\gamma_1$ (Neumann trace operator). We assume for simplicity that $g_j$ have compact support in $(0,T)$ for $j=0,1$. Note that this in particular implies compatibility at the space time corner point. We will write $y(t)=W_b[g_0,g_1](t)$ for the solution of the above PDE, where $W_b$ denotes the \texttt{BdIntOp} that we will construct by using the \texttt{UTM}.

A representation formula for solution of an easier problem, where $P=-i\partial_x^4$ (without the Laplacian) was recently obtained in a recent work of first author \cite{OY19}.  In that work, the \texttt{BdIntOp} was found via the UTM in the form
 \begin{equation}\label{ysol}
W_b[g_0,g_1](x,t)=\int_{\partial D^+}E(k;x,t)G(k;T)dk,
\end{equation} where $$D^+:=\left\{k\in\mathbb{C}\,|\,\text{arg}\,k\in \bigcup_{\ell=1}^2\left(\frac{(2\ell-1)\pi}{4},\frac{\ell\pi}{2}\right)\right\},$$
$$E=-\frac{1}{2\pi}e^{ikx+ik^4t},$$ $$G(k;t)=-2ik(k+\nu(k))\tilde{g}_1(-ik^4,t)-2k\nu(k)(k+\nu(k))\tilde{g}_0(-ik^4,t)$$ with
$$\nu(k)=\left\{  \begin{array}{ll} ik, & \hbox{$\arg k\in \{\frac{\pi}{4},\frac{\pi}{2}$\};} \\   -ik, & \hbox{$\arg k\in  \{\frac{3\pi}{4},\pi\},$} \end{array}  \right.$$ and
\begin{equation}\label{ytildej}\tilde{g}_j(k,t):=\int_{0}^{t}e^{ks}g_j(s)ds.\end{equation} 

There is another study (see \cite{Cap2020}) in which a \texttt{BdIntOp} corresponding to the biharmonic case $P=-i\partial_x^4$ is constructed. In their paper authors use a Riemann–Liouville fractional integral.  This method is well known and was previously used for the Korteweg-de Vries (KdV) equation by Colliander and Kenig \cite{Col} and later for the Schrödinger equation by Holmer \cite{Hol}. To the best of our knowledge Riemann–Liouville fractional integral method was used for PDEs that involved only a single spatial derivative term.  It is also possible to use the Laplace transform in time to construct a \texttt{BdIntOp}, a method which was nicely applied both for the Schrödinger equation \cite{Bona18} and the KdV equation \cite{Bon1} by Bona, et. al.  Laplace transform method is an effective method in general but the technical analysis of solutions gets more difficult if the order of PDE is high and there are multiple spatial derivative terms.  This is because one has to deal with higher order characteristic equations to be able solve an infinite family of higher order ODEs, an algebraic difficulty. In addition, inverting the associated Bromwich integral is another challenge for such PDEs because a subtle singularity analysis must be performed.

An alternative which bypasses issues of the approaches mentioned in the above paragraph is the Fokas method \cite{MixedDer, OY19}. It is worth mentioning that even with this method there are some difficulties for the current problem.  The challenge here is that in this more general setting, where $P=-i(\partial_x^4+\partial_x^2)$, certain analyticity issues arise related with the third step of the \texttt{UTM}.  Observe that, in the case $P=-i\partial_x^4$, the spectral input of boundary terms is $w(k)=-ik^4$. Therefore, there are nontrivial entire (analytic on $\mathbb{C}$) maps such as $k\mapsto \mp ik$, $k\longmapsto -k$, which keep the spectral input invariant.  Existence of such nice maps play an important role in the contour deformation and elimination of unknowns from the formula of the \texttt{BdIntOp}.  On the other hand, the spectral input of boundary terms turn out to be $w(k)=-i(k^4-k^2)$ if $P=-i(\partial_x^4+\partial_x^2)$. It is not clear whether there exists a map $k\mapsto \nu(k)$ which satisfies the invariance property $w(\nu(k))=w(k)$, namely $-i(\nu^4(k)-\nu^2(k))=-i(k^4-k^2)$ and is also analytic  on a union of simply connected open sets, each of which contains the region whose boundary is part of the standard (deformed) contour of integration used in the \texttt{UTM}.

The above technical issue may arise in most higher order PDEs where the main differential operator is a linear combination of more than one term.  An example is the Korteweg-de Vries (KdV) equation \cite{Decon14}. Another context for observing this analyticity issue is higher dimensional PDEs which involve mixed derivatives \cite{MixedDer}.  It can even happen in second order PDEs with a second order time derivative such as the wave equation \cite{Decon18}.  This analyticity issue stems from the use of complex root functions, which are typically discontinuous, to construct invariance maps.  Recently, \cite{MixedDer} recommended rotating the branch cut for the square root function to a suitable degree and proved that this moves the domain of nonanaliticity of the invariance maps away from the desired contour of integration, except at a single branch point which does not affect the relevant analysis. In this work, we follow a similar approach for constructing the \texttt{BdIntOp} associated with \eqref{maineq}-\eqref{bdry}. In the last section, we present the \texttt{BdIntOp} for the class of fourth order Schrödinger operators given in the form $P = -i(\alpha\partial_x^4+\beta\partial_x^2)$, where $0\neq\alpha\in\mathbb{R}$ and $\beta\in \mathbb{R}$.

The main result of the paper is given in Theorem \ref{mainthm}. In this direction, we utilized the nice space-time dependence, i.e., oscillatory nature of the \texttt{BdIntOp} \eqref{bdintop2} for proving the desired dispersion estimates; results on the whole space proved in \cite{Artzi2000} were also used.

\section{Construction of the boundary integral operator}\label{seccon}
In this section, we construct the \texttt{BdIntOp} associated with \eqref{maineq}-\eqref{bdry}.  To this end, we will first assume that $u$ is sufficiently smooth in $\Omega_{T}=\mathbb{R}_+\times (0,T)$ up to the boundary of $\Omega_T$,  and also that $u$ decays sufficiently fast as $x\rightarrow \infty$. Once the \texttt{BdIntOp} is constructed, then the smoothness condition can be given up as the integral will still make sense under much weaker assumptions on data. In particular, the integral formula will serve as the definition of a \emph{weak} solution. In order to obtain a global relation (the first ingredient of the UTM), we introduce the half line Fourier transform:
\begin{equation}\label{yhatkt}
  \hat{y}(k,t) \equiv \int_0^\infty e^{-ikx}y(x,t)dx, \quad \Im k\le 0.
\end{equation} Note that the condition $\Im k\le 0$ is essential for the convergence of the above integral.  We also introduce the functions $\tilde{g}_j$ defined by the formula \eqref{ytildej} for $0\le j\le 2$, so called $t-$\emph{transforms} of boundary traces, some of which are unknown such as $t-$\emph{transforms} of $g_j(t):=\partial_x^jy(0,t)$, $j=2,3$.
Taking the half line Fourier transform of \eqref{maineq}-\eqref{bdry} and integrating the resulting ordinary differential equation in time, we obtain the global relation
\begin{multline}\label{IntythatktRe}
e^{w(k)t}\hat{y}(k,t)=-i\tilde{g}_3(w(k),t)+k\tilde{g}_2(w(k),t)\\
-i(1-k^2)\tilde{g}_1(w(k),t)+k(1-k^2)\tilde{g}_0(w(k),t), \quad \Im\,k\le 0,
\end{multline}
with
\begin{equation}\label{def:w}
w(k)=-i(k^4-k^2).
\end{equation}
Taking the inverse Fourier transform, we find that $u$ must satisfy
\begin{equation}\label{qxt}
y(x,t)=\int_{-\infty}^{\infty}E(k;x,t)\tilde{g}(w(k),t)dk, \quad x\in\mathbb{R_+}, t>0,
\end{equation} where
\begin{equation}\label{def:E}
E(k;x,t)=-\frac{1}{2\pi}e^{ikx-w(k)t}
\end{equation}
 and $\tilde{g}=i\tilde{g}_3-k\tilde{g}_2+(1-k^2)(i\tilde{g}_1-k\tilde{g}_0).$  Since only the Dirichlet and Neumann boundary values are known, the values $\tilde{g}_2$ and $\tilde{g}_3$ are unknowns in the above formulation.  In order to eliminate these unknown boundary terms from \eqref{qxt}, the first step is deforming the integral on the real line to a more suitable contour in the upper half complex plane. To this end, we first define the following regions:
\begin{equation}\label{defD+} D\equiv \{k\in\mathbb{C}\,|\,\Re(w(k))<0\},\quad D^+=D\cap \mathbb{C}_+, D^-=D\cap \mathbb{C}_{-}.
\end{equation} Note that in $\mathbb{C}\setminus D^+$, the exponential $e^{-w(k)(t-s)}$ is bounded.  Therefore, the term $E\tilde{g}$ is analytic and decays as $k\rightarrow \infty$ for $k\in \mathbb{C}\setminus D^+$.  Thus, by using Cauchy's theorem and Jordan's lemma, we can rewrite \eqref{qxt} in the form
\begin{equation}\label{qxtD+}
y(x,t)=\int_{\partial D^+}E(k;x,t)\tilde{g}(w(k),t)dk, \quad x\in\mathbb{R_+}, t>0,
\end{equation} where the orientation is so that $D^+$ stays at the left of $\partial D^+$ as the contour is traversed.

The second step for eliminating unknowns is the use of invariance maps, i.e., maps that keep the spectral input $w(k)$ unchanged. By definition, such a map must satisfy $\nu^4(k)-\nu^2(k)=k^4-k^2$, which is equivalent to $$(\nu(k)-k)(\nu(k)+k)(\nu^2(k)+k^2-1)=0.$$

It follows that one nontrivial invariance map is $k\mapsto -k$. Using this transformation, we can rewrite the global relation \eqref{IntythatktRe} as
\begin{multline}\label{Intythatkt-b}
 e^{w(k)t}\hat{y}(-k,t)=-i\tilde{g}_3(w(k),t)-k\tilde{g}_2(w(k),t)\\
-i(1-k^2)\tilde{g}_1(w(k),t)
-k(1-k^2)\tilde{g}_0(w(k),t), \quad \Im k\ge 0.
\end{multline}

Furthermore, changing $k$ by an invariance map $\nu(k)$ satisfying \begin{equation}\label{invprop}\nu^2(k)=1-k^2\end{equation} in \eqref{Intythatkt-b}, we can rewrite the global relation in the form
\begin{multline}\label{Intythatkt-c}
  e^{w(k)t}\hat{y}(-\nu(k),t)=-i\tilde{g}_3(w(k),t)-\nu(k)\tilde{g}_2(w(k),t)\\
  -ik^2\tilde{g}_1(w(k),t)-\nu(k)k^2\tilde{g}_0(w(k),t), \quad \Im\nu(k)\ge 0.
\end{multline}

Using \eqref{Intythatkt-b} and \eqref{Intythatkt-c}, we have
\begin{multline}\label{Intythatkt-d}
  -k\tilde{g}_2(w(k),t)= -ke^{w(k)t}\frac{(\hat{y}(-k,t)-\hat{y}(-\nu(k),t))}{\nu(k)-k}\\
  -ik(\nu(k)+k)\tilde{g}_1(w(k),t)-k^2\nu(k)\tilde{g}_0(w(k),t)
\end{multline} and
\begin{multline}\label{Intythatkt-e}
  i\tilde{g}_3(w(k),t)=-\frac{ \nu(k)\hat{y}(-k,t)-k\hat{y}(-\nu(k),t)}{\nu(k)-k} e^{w(k)t}\\
  -i\left(k^2+\nu^2(k)+k\nu(k)\right)\tilde{g}_1(w(k),t)-k\nu(k)(k+\nu(k))\tilde{g}_0(w(k),t)
\end{multline} provided $\Im k\ge 0$ and $\Im\nu(k)\ge 0$.

Now, we can rewrite \eqref{qxtD+} in the form
\begin{equation}\label{qxtD+new}
  y(x,t)=\int_{\partial D^+}E(k;x,t)G(k;t)dk +\frac{1}{2\pi}\int_{\partial D^+}e^{ikx}H(k;t)dk,  \quad x\in\mathbb{R_+}, t>0,
\end{equation}
where
\begin{equation}\label{def:G}
G(k;t)=-2ik(k+\nu(k))\tilde{g}_1(w(k),t)-2k\nu(k)(k+\nu(k))\tilde{g}_0(w(k),t)
\end{equation}
 and
$$H(k;t)=\frac{ \nu(k)+k}{\nu(k)-k}\hat{y}(-k,t)-\frac{2k}{\nu(k)-k}\hat{y}(-\nu(k),t).$$
Observe that $H$ becomes singular at $k\in\mathbb{C}$ if $\nu(k)=k$.  This can only be true if $k=\mp \frac{1}{\sqrt{2}}$ due to the invariance property \eqref{invprop}. In either case, this would be only a removable singularity if we knew $\nu$ were analytic.  In that case, we could easily conclude that the second integral in \eqref{qxtD+new} is zero. However, there is no map which both satisfies the invariance property \eqref{invprop} and is for instance analytic in the neighborhood of $k=-1\in \overline{D^+}$. This can be proven by using arguments similar to those in the proof of \cite[Lemma 4.2]{MixedDer}.  In order to deal with this analyticity issue, we introduce a square root function whose branch cut is slightly rotated compared to the standard square root function. We set $\displaystyle \sqrt{z}^*:=|z|^{\frac{1}{2}}e^{i\frac{\arg z}{2}}$ with $\arg z\in [\epsilon,2\pi+\epsilon)$ for some fixed and sufficiently small $\epsilon>0$, and choose
\begin{equation}\label{nukmod}
\nu(k)=\sqrt{1-k^2}^*.
\end{equation}
Then, $\nu$ satisfies the invariance property \eqref{invprop} and is analytic on $\overline{D^+}\setminus \{-1\}$ (See Figure \ref{Branchcut}).

Moreover, $\Im \nu(k)\ge 0$ for all $k\in \overline{D_+}$.  The discontinuity point $k=-1$ can be taken care of by using the same complex analytic arguments given in \cite[Section 4, pg. 13]{MixedDer}.  In more details, we remove a small half ball $B_r$ from $D^+$ around the branch point $k=-1$ and show via Cauchy's theorem and Jordan's lemma that the integral around the boundary of $D^+\setminus B_r$ vanishes as $H$ is analytic and bounded in this region. Moreover, the integral around $\partial B_r$ vanishes as $r\rightarrow 0$ since $H$ is bounded on $B_r$ (even if it is not analytic).  In conclusion, we justify that the second integral in \eqref{qxtD+new} is zero.
\begin{figure}[h]
 \centering
 \includegraphics[width=0.7\textwidth]{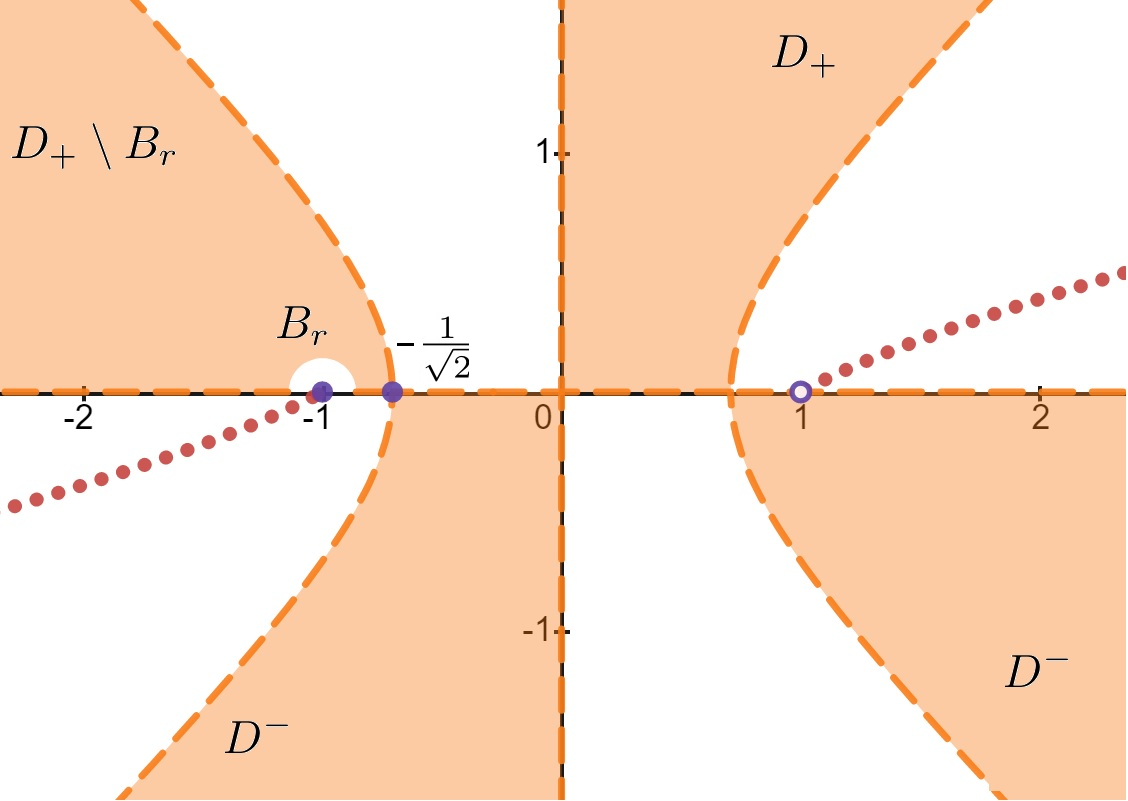}
  \caption{Red path denotes the branch cut of $\nu(k)=\sqrt{1-k^2}^*$}\label{Branchcut}
\end{figure}

Hence, the \texttt{BdIntOp} associated with \eqref{maineq}-\eqref{bdry} is given by
\begin{equation}\label{bdintop2}
  W_b[g_0,g_1](x,t) = \int_{\partial D^+}E(k;x,t)G(k;t)dk,
\end{equation}
where $\nu$ is defined in \eqref{nukmod}, $E$ is defined in \eqref{def:E} and $G$ is defined in \eqref{def:G}. One can of course replace $G(k;t)$ at the right hand side of \eqref{bdintop2} with $G(k;T)$ by using the standard arguments in the Fokas method. Therefore, we have the formula given in the theorem below:
\begin{thm}[Integral representation]\label{rep}Suppose $y$ solves \eqref{maineq}-\eqref{bdry} in $\Omega_{T}=\mathbb{R}_+\times (0,T)$, is sufficiently smooth up to the boundary of $\Omega_T$  and decays sufficiently fast as $x\rightarrow \infty$, uniformly in $t\in [0,T]$.  Then, the associated \texttt{BdIntOp} is defined by
  \begin{equation}\label{bdintop2T}
  W_b[g_0,g_1](x,t) = \int_{\partial D^+}E(k;x,t)G(k;T)dk,
\end{equation} where $E$ and $G$ are given by \eqref{def:E} and \eqref{def:G}, respectively and $\partial D^+$ is the boundary of the region $D^+$ defined in \eqref{defD+} with orientation that $D^+$ remains at the left of $\partial D^+$ as the boundary is traversed.
\end{thm}
  The advantage of the above form with $T$ in \eqref{bdintop2T} relative to the formulation in \eqref{bdintop2} is that differentiation with respect to space and time is very straightforward since it only affects the exponential term $E(k;x,t).$ This is important for interpolation arguments because an estimate at the base level can then be extended to higher regularity levels via differentiating in $x$ and applying the base level estimate again.

The main result concerning the spatial norms associated with the \texttt{BdIntOp} is below:
\begin{thm}[Dispersion estimates]\label{mainthm}
  Let $W_b$ be the \texttt{BdIntOp} defined by \eqref{bdintop2T}.  Then, it satisfies the following estimate
  \begin{equation}\label{finalest}
\left\|W_b[g_0,g_1]\right\|_{L_x^r(\mathbb{R}_+)} \lesssim t^{-(\frac{1}{4}-\frac{1}{2r})}\sum_{i=1}^5\|\Psi_i\|_{L^{r'}}, \qquad 0<t\leq 1,
\end{equation} for $r\in [2,\infty]$, where $\Psi_i$, $i=\overline{1,5}$ are defined in \eqref{psi1}, \eqref{psi2}, \eqref{FTphi3}, \eqref{FTphi4}, \eqref{psi5}, respectively in terms of given Dirichlet-Neumann data $(g_0, g_1)$.
\end{thm}

The dispersion estimates found above imply $L_t^\lambda L_x^r(\mathbb{R}_+)$ type Strichartz estimates with respect to $L^2$ norm of $\Psi_i$, $i=1,2,3,4$ for suitable, i.e., biharmonic admissible, $(\lambda, r)$, i.e., $\frac{1}{8}=\frac{1}{4r}+\frac{1}{\lambda}$, $\lambda,r\in[2,\infty]$. Observe that the representation formula is very  favorable  for differentiating with respect to $x$ and each derivative merely brings a factor of $k$ into the integrand. Therefore, one can differentiate and obtain $L_t^\lambda W_x^{s,r}(\mathbb{R}_+)$ type estimates with respect to  $H^s$ norms of $\Psi_i$ at first for integer $s$ and then by interpolation for fractional $s$.  Finally, it is not difficult to show by using the Fourier characterization of Sobolev norms that $H^s$ norms of $\Psi_i$ are controlled by $H_t^{\frac{2s+3}{8}}(0,T)$ and  $H_t^{\frac{2s+1}{8}}(0,T)$ norms of boundary data $g_0$ and $g_1$, respectively. See for instance \cite{Ozsari21Str1} for such arguments in the case of the Schrödinger equation. Therefore, we have the corollary below whose proof can be done by using the dispersion estimate in Theorem \ref{mainthm} and slightly modifying the arguments given in \cite{Ozsari21Str1} for the classical Schrödinger equation.

\begin{cor}[Strichartz estimates]\label{cor1}
	Let $s\ge0$, $T\le 1$, $g_0\in H_t^{\frac{2s+3}{8}}$, $g_1\in H_t^{\frac{2s+1}{8}}$ with $supp\,g_0, supp\,g_1 \subset [0,T)$ and $(\lambda,r)$ be biharmonic admissible and $W_b$ be the \texttt{BdIntOp} defined by \eqref{bdintop2T}. Then, $W_b[g_0,g_1]$ defines an element of $C([0,T];H^s(\mathbb{R}_+))$ that satisfies the following inhomogeneous Strichartz estimate:
	\begin{equation}\label{Strnew}|W_b[g_0,g_1]|_{L_t^\lambda(0,T; W_x^{s,r}(\mathbb{R}_+))}\lesssim |g_0|_{H_t^{\frac{2s+3}{8}}(\mathbb{R})}+|g_1|_{H_t^{\frac{2s+1}{8}}(\mathbb{R})},\end{equation} where the constant of the inequality depends on $s$.
\end{cor}

\subsection*{Further implications}

\begin{itemize}
	\item 	The dispersion estimate and the Strichartz estimate in this paper also hold in the easier case where the pde only involves the biharmonic operator and does not involve the Laplacian. Moreover, in the purely biharmonic case the restriction $t\le 1$ in \eqref{finalest} and the condition on $T$ in Corollary \ref{cor1} can both be removed. This is because the analog of the oscillatory estimate in Theorem \ref{thm1} does not require the restriction $0<t\le 1$ when the oscillatory term in the integral does not involve the exponent $s^2$ associated with the Laplacian, see \cite{Artzi2000}.\\
	
	\item The results of Theorem \ref{mainthm} and Corollary \ref{cor1} are also useful for treating the corresponding nonlinear problems.  Recently, \cite{Guo21} studied the local wellposedness for the nonlinear fourth order Schrödinger equation posed on the half line with inhomogeneous Dirichlet-Neumann boundary conditions.  The authors obtained local wellposedness in the high regularity setting, namely for $s>1/2$. The problem remains open in the low regularity setting $0\le s<1/2$ which is a more difficult problem even for power type nonlinearities such as $u\mapsto |u|^pu$, $p>0$. This is because the space $H^s(\mathbb{R}_+)$ is no longer a Banach algebra for $s>1/2$.  The classical tool for treating this difficulty is using Strichartz estimates.  Therefore, the Strichartz estimate in Corollary \eqref{cor1} can be considered as a first step towards establishing local wellposedness in the low regularity setting for the associated nonlinear models. Of course, in addition to the boundary type Strichartz estimates established here, one also needs to prove time trace estimates in fractional Sobolev spaces for the homogeneous and nonhomogeneous linear Cauchy problems to be able to fully treat the nonlinear problem.  Proving Strichartz estimates for the homogeneous and nonhomogeneous Cauchy problems is not difficult and can be done by modifying the well known arguments for the classical Schrödinger equation.  However, the time trace analysis of solutions of the nonhomogeneous Cauchy problem is a quite challenging problem at the low regularity setting. We leave this problem as well as the full treatment of the nonlinear model as a future work.\\
	
	\item   The results of this section extend to more general type of fourth order differential operators in the form  $P = -i(\alpha\partial_x^4+\beta\partial_x^2)$, where $0\neq\alpha\in\mathbb{R}$ and $\beta\in \mathbb{R}$. We show in the last section that $D^+$ and the correct choice of a branch cut for the square root function in the definition of invariance maps change depending on the signs and values of $\alpha$ and $\beta$.
\end{itemize}

\section{Dispersion estimates - Proof of Theorem \ref{mainthm}} In this section, we prove that dispersion estimates for the fourth order Schrödinger equation posed on the whole line (see e.g., \cite{Artzi2000}) can be extended to the case of half line and in particular one can obtain boundary smoothing properties associated with \texttt{BdIntOp} \eqref{bdintop2} . To this end, we first observe that the integral on $\partial D^+$ is equivalent to the sum of integrals over the union of paths given by
\begin{align}
&\gamma_1(s)=is,\quad \quad \infty > s \geq 0, \label{gamma1}\\
&\gamma_2(s)=s,\quad \quad 0< s < \tfrac{1}{\sqrt{2}}, \label{gamma2}\\
&\gamma_3(s)=s+i(s^2-\tfrac{1}{2})^{\tfrac{1}{2}},\quad \tfrac{1}{\sqrt{2}}\leq s < \infty , \label{gamma3}\\
&\gamma_4(s)=-s+i(s^2-\tfrac{1}{2})^{\tfrac{1}{2}} ,\quad \tfrac{1}{\sqrt{2}}\leq s < \infty, \label{gamma4}\\
&\gamma_5(s)=s,\quad\quad -\infty < s \le -\tfrac{1}{\sqrt{2}}. \label{gamma5}
\end{align}


We first split the representation formula in five pieces according to the above paths: \begin{equation}\label{Wbdecomp}
                     W_b[g_0,g_1](x,t) = \sum_{\ell=1}^5\int_{\gamma_\ell}E(k;x,t)G(k;T)dk=: W_b^\ell[g_0,g_1](x,t).
                   \end{equation}
We will find estimates for each of the terms at the right hand side of \eqref{Wbdecomp}.

For $\ell=1$, we have
\begin{equation}\label{Wb1rewrite}
\begin{split}
   W_b^1[g_0,g_1](x,t)=\int_{\gamma_1}E(k;x,t)G(k;T)dk\\
   =\frac{i}{2\pi}\int_0^\infty e^{-sx+i(s^4+s^2)t}G(is,T)ds.
\end{split}
\end{equation}
Let $\Psi_1$ is defined to be the inverse Fourier transform of $\hat{\Psi}_1$, where
\begin{equation}\label{psi1}
  \hat{\Psi}_1(s)=G(is,T) \text{ for }s\ge 0\text{ and }\hat{\Psi}_1(s)=0\text{ for }s< 0.
\end{equation}
 Then,
\begin{equation}\label{Wb1rewrite2}
  \begin{split}
   W_b^1[g_0,g_1](x,t)=\frac{1}{2\pi}\int_{0}^\infty e^{-sx+i(s^4+s^2)t}\int_{-\infty}^\infty e^{-isy}\Psi_1(y) dyds.\end{split}
\end{equation}
By changing the order of integration, we can represent $W_b^1[g_0,g_1](x,t)$ as
\begin{equation}\label{Wb1rewrite3}
\begin{split}
W_b^1[g_0,g_1](x,t)&=\frac{1}{2\pi}\int_{-\infty}^\infty \left[\int_{0}^\infty e^{-sx+i(s^4+s^2)t-isy}ds\right] \Psi_1(y)dy\\
&=\frac{1}{2\pi}\int_{-\infty}^\infty K_1(y;x,t) \Psi_1(y)dy,
\end{split}
\end{equation}
where $K_1(y;x,t)$ is called the kernel of $W_b^1$ and given by
\begin{equation}\label{w1kernel}
\begin{split}
K_1(y;x,t)=\int_{0}^\infty e^{-sx+i(s^4+s^2)t-isy}ds=\int_{0}^\infty e^{i\phi(s;y,t)}\mathit{p}(s,x) ds
\end{split}
\end{equation}
with the amplitude function $\mathit{p}(s,x)=e^{-sx}$ and the phase function $$\phi(s;y,t)=(s^4+s^2)t-sy.$$  We have the following lemma.

\begin{lem}\label{van}
  Let $$I(s;y,t) \equiv \int_{0}^{s}e^{i(\xi^4+\xi^2)t-i\xi y}d\xi.$$ Then, $$|I(s;y,t)|\le ct^{-1/4},$$ where $c>0$ is independent of $y\in\mathbb{R}$ and $t,s>0$.
\end{lem}

\begin{proof}
	See Appendix \ref{ProofofLemma31}.
\end{proof}

\begin{rem}
  Note that in the above lemma the interval of integration is finite but the constant of the inequality is independent of the upper limit $s$ which is a crucial ingredient for the next lemma below. The unbounded case where the interval of integration is the whole line is given below in Theorem \ref{thm1} and due to Ben-Artzi, Koch, Saut \cite{Artzi2000}. The unbounded case is critical in the analysis of $W_b^2$.
\end{rem}

 \begin{lem}\label{k1lem}
 	The kernel of $W_b^1$ defined by $\eqref{w1kernel}$ satisfies the following dispersive estimate:
 	\begin{equation}
 	\begin{split}
 	\abs{K_1(y;x,t)}\lesssim t^{-1/4},
 	\end{split}
 	\end{equation}
 	where $x,t\in \mathbb{R}_{+}$ and $y \in\mathbb{R}$.
 \end{lem}
\begin{proof}
	We first set $\Phi(s;y,t)\equiv I(s;y,t)$. Then write the kernel
	\begin{equation}\label{w1Krewrite}
K_1(y;x,t)=\int_{0}^\infty \left[ \frac{d}{ds}\Phi(s;y,t)\right] p(s,x)ds.
	\end{equation}
	Integrating by parts at the RHS of $\eqref{w1Krewrite}$ and using $$\lim_{s\rightarrow \infty}\Phi(s;y,t)p(s;x)=0$$ we find
	\begin{equation}\label{w1esK}
	\abs{K_1(y;x,t)}\leq \int_{0}^{\infty} \abs{\Phi(s;y,t)}\abs{\frac{d}{ds}p(s,x)}ds.
	\end{equation}
	By Lemma $\eqref{van}$, we have $\abs{\Phi(s;y,t)}\lesssim t^{-1/4}$. Therefore,
	\begin{equation}
    \begin{split}
	\abs{K_1(y;x,t)}\le c t^{-1/4}\int_0^\infty\abs{\frac{d}{ds}p(s,x)}ds=ct^{-1/4}\left(x\int_0^\infty e^{-sx}ds\right)\\
=ct^{-1/4}(1-e^{-sx})\leq ct^{-1/4}.
\end{split}
	\end{equation}
\end{proof}
The following estimate is deduced from Lemma \ref{k1lem}:
\begin{equation}\label{wb1linfty}
  \left\|W_b^1[g_0,g_1]\right\|_{L_x^\infty(\mathbb{R}_+)}\lesssim t^{-\frac{1}{4}}\|\Psi_1\|_{L^1}, t>0.
\end{equation}
On the other hand, using \eqref{Wb1rewrite} and the boundedness of Laplace transform, we have
\begin{equation}\label{wb1l2}
\begin{split}
 \left\|W_b^1[g_0,g_1]\right\|_{L_x^2(\mathbb{R}_+)}^2=\frac{1}{(2\pi)^2}\int_0^\infty\left|\int_0^\infty e^{-sx+i(s^4+s^2)t}G(is,T)ds\right|^2dx\\
\lesssim \int_0^\infty\left(\int_0^\infty e^{-sx}|G(is,T)|ds\right)^2dx\lesssim \int_0^\infty |G(is,T)|^2ds\\
=\int_{-\infty}^\infty |\widehat{\Psi_1}(s)|^2ds=\|\Psi_1\|_{L^2}^2.
\end{split}
\end{equation}
Interpolating between \eqref{wb1linfty} and \eqref{wb1l2}, we obtain
\begin{equation}\label{wb1inter}
  \left\|W_b^1[g_0,g_1]\right\|_{L_x^r(\mathbb{R}_+)} \lesssim t^{-(\frac{1}{4}-\frac{1}{2r})}\|\Psi_1\|_{L^{r'}}
\end{equation} for $r\in [2,\infty]$.

Regarding the case $\ell=2$, we first define $\Psi_2$ to be the inverse Fourier transform of $\hat{\Psi}_2$, where
\begin{equation}\label{psi2}
\hat{\Psi}_2(s)=G(s,T)\text{ for }\frac{1}{\sqrt{2}}\ge s\ge 0\text{ and }\hat{\Psi}_2(s)=0,\text{ otherwise}.
\end{equation}
Now, we can extend our limits of integral to the whole real line:
\begin{equation}\label{Wb2rewrite2}
  \begin{split}
   W_b^2[g_0,g_1](x,t)=-\frac{1}{2\pi}\int_{0}^\frac{1}{\sqrt{2}} e^{isx+i(s^4-s^2)t}\hat{\Psi}_2(s)ds\\
=-\frac{1}{2\pi}\int_{-\infty}^\infty\Psi_2(y)\int_{-\infty}^\infty e^{isx+i(s^4-s^2)t-isy}dsdy
=:\int_{-\infty}^\infty\Psi_2(y)K_2(y;x,t)dy.
\end{split}
\end{equation}
To estimate the kernel we use \cite[Theorem 1]{Artzi2000}:
\begin{thm}(\cite{Artzi2000})
	\label{thm1}
	Let $t\leq 1$ or $\abs{x}\geq t$ and consider the oscillatory integral
	\begin{equation}
		{I}(x,t)=\int_{\mathbb{R}}e^{it({s}^4- {s}^2)+ixs}ds.
	\end{equation}
	Then,
	\begin{equation}
	\abs{{I}(x,t)}\leq ct^{-\tfrac{1}{4}}\left(1+\tfrac{\abs{x}}{t^{1/4}}\right)^{-\tfrac{1}{3}}.
	\end{equation}
\end{thm}

By using Theorem \ref{thm1}, we find the decay estimate for kernel $K_2(x,y,t)$:
 \begin{equation}
 |K_2(x,y,t)|\lesssim t^{-\frac{1}{4}}\left(1+\frac{\abs{x-y}}{t^{1/4}}\right)^{-1/3}.
 \end{equation}
Since the term $\left(1+\frac{\abs{x-y}}{t^{1/4}}\right)^{-1/3}\leq 1$, we get the desired estimate for the kernel:
\begin{equation}
 |K_2(x,y,t)|\lesssim t^{-\frac{1}{4}}.
\end{equation}
The above estimate implies
\begin{equation}\label{wb2linfty}
  \left\|W_b^2[g_0,g_1]\right\|_{L_x^\infty(\mathbb{R}_+)}\lesssim t^{-\frac{1}{4}}\|\Psi_2\|_{L^1}.
\end{equation}
On the other hand extending \eqref{Wb2rewrite2} to $x\in \mathbb{R}$, we obtain that
\begin{equation}\label{Wb2l2}
  \mathcal{F}(W_b^2[g_0,g_1])(s,t)= -e^{i(s^4-s^2)t}\hat{\Psi}_2(s),
\end{equation} which gives
\begin{equation}\label{wb2l2}
\begin{split}
 \left\|W_b^2[g_0,g_1]\right\|_{L_x^2(\mathbb{R}_+)}\le\left\|W_b^2[g_0,g_1]\right\|_{L_x^2(\mathbb{R})}= \|\Psi_2\|_{L^2}.
\end{split}
\end{equation}
Interpolating between \eqref{wb2linfty} and \eqref{wb2l2}, we obtain
\begin{equation}\label{wb1inter}
  \left\|W_b^2[g_0,g_1]\right\|_{L_x^r(\mathbb{R}_+)} \lesssim t^{-(\frac{1}{4}-\frac{1}{2r})}\|\Psi_2\|_{L^{r'}}
\end{equation} for $r\in [2,\infty]$.

For $\ell=5$, we have
\begin{equation}
\begin{split}
W_b^5[g_0,g_1](x,t)=\int_{\gamma_5}E(k;x,t)G(k;T)dk\\
=-\frac{1}{2\pi}\int^{-\tfrac{1}{\sqrt{2}}}_{-\infty} e^{isx+i(s^4-s^2)t}G(s,T)ds.
\label{Wb5}
\end{split}
\end{equation}
We set $\Psi_5$ to be the inverse Fourier transform of $\hat{\Psi}_5(s)$, where
\begin{equation}\label{psi5}
\hat{\Psi}_5(s):=G(s,T)\text{ for
 }s\leq - \frac{1}{\sqrt{2}}\text{ and }\hat{\Psi}_5(s)=0\text{ for }s> -\frac{1}{\sqrt{2}}.
\end{equation}
So we  can rewrite the fifth component of the \texttt{BdIntOp} in the following form:
\begin{equation} \label{Wb5rewrite}
\begin{split}
W_b^5[g_0,g_1](x,t)=-\frac{1}{2\pi}\int_{-\infty}^\infty\Psi_5(y)\int^{-\tfrac{1}{\sqrt{2}}}_{-\infty} e^{isx+i(s^4-s^2)t-isy}dsdy\\
=:\int_{-\infty}^{\infty}\Psi_5(y)K_5(x,y,t)dy.
\end{split}
\end{equation}
By similar calculations that we used for $W_b^2$, we have
\begin{equation}\label{Estc5}
|K_5(x,y,t)|\lesssim t^{-\tfrac{1}{4}}.
\end{equation}
Using $\eqref{Estc5}$ in $\eqref{Wb5rewrite}$, we have
\begin{equation}\label{Wb5Linfty}
\left\|W_b^5[g_0,g_1]\right\|_{L_x^\infty(\mathbb{R}_+)}\lesssim t^{-\tfrac{1}{4}}\|\Psi_5\|_{L^1}.
\end{equation}
On the other hand extending \eqref{Wb5rewrite} to $x\in \mathbb{R}$, we get:
\begin{equation}\label{Wb5l2}
\mathcal{F}(W_b^5[g_0,g_1])(s,t)= -e^{i(s^4-s^2)t}\hat{\Psi}_5(s),
\end{equation} which gives
\begin{equation}\label{Esw5}
\begin{split}
\left\|W_b^5[g_0,g_1]\right\|_{L_x^2(\mathbb{R}_+)}\le\left\|W_b^5[g_0,g_1]\right\|_{L_x^2(\mathbb{R})}= \|\Psi_5\|_{L^2}.
\end{split}
\end{equation}
Interpolating between $\eqref{Esw5}$ and $\eqref{Wb5Linfty}$, we obtain
\begin{equation}\label{wb5inter}
\left\|W_b^5[g_0,g_1]\right\|_{L_x^r(\mathbb{R}_+)} \lesssim t^{-(\frac{1}{4}-\frac{1}{2r})}\|\Psi_5\|_{L^{r'}}
\end{equation}for $r\in [2,\infty]$.

For $\ell=3$, we have
\begin{equation}\label{Wb3}
\begin{split}
&W_b^3[g_0,g_1](x,t)=\int_{\gamma_3}E(k;x,t)G(k;T)dk\\
&=\frac{1}{2\pi}\int_{\tfrac{1}{\sqrt{2}}}^\infty e^{isx-(s^2-\tfrac{1}{2})^{\tfrac{1}{2}} x-i(2s^2-\tfrac{1}{2})^2t}G(s+i(s^2-\tfrac{1}{2})^{\tfrac{1}{2}},T)(1+\tfrac{is}{(s^2-\tfrac{1}{2})^{\tfrac{1}{2}}})ds.
\end{split}
\end{equation}
Let $\Psi_3$ be defined as the inverse Fourier transform of
\begin{equation} \label{FTphi3}
\hat{\Psi}_3(s)=\left\{
\begin{array}{ll}
	G(s+i(s^2-\tfrac{1}{2})^{\tfrac{1}{2}},T)\Bigg(1+\tfrac{is}{(s^2-\tfrac{1}{2})^{\tfrac{1}{2}}}\Bigg), & s\ge \tfrac{1}{\sqrt{2}}\\
	0 & s< \tfrac{1}{\sqrt{2}}.
\end{array} \right.
\end{equation}
By changing the order of the integration we  can rewrite $W_b^3$ in the following form:
\begin{equation}\label{Wb3rewrite}
\begin{split}
W_b^3[g_0,g_1](x,t)&=-\frac{1}{2\pi}\int_{-\infty}^\infty\Psi_3(y)\int_{\tfrac{1}{\sqrt{2}}}^\infty e^{isx-(s^2-\tfrac{1}{2})^{\tfrac{1}{2}} x-i(2s^2-\tfrac{1}{2})^2t-isy}dsdy\\
&=:\int_{-\infty}^{\infty}K_3(y;x,t)\Psi_3(y)dy.
\end{split}
\end{equation}
where $K_3(y;x,t)$ is the kernel of $W_b^3$. Now, we can show that $K_3$ decays as ${t}^{-1/4}$ by using a similar analysis that was given for $W_b^1$. Indeed, we can write

\begin{equation}\label{DefK3}
\begin{split}
K_3(y;x,t)&=\int_{\tfrac{1}{\sqrt{2}}}^\infty e^{is(x-y)-i(4s^4-2s^2+\frac{1}{4})t-(s^2-\tfrac{1}{2})^{\tfrac{1}{2}} x}ds\\
&=\int_{\tfrac{1}{\sqrt{2}}}^\infty e^{i\phi_3(s;x,y,t)}p_3(s,x)ds
\end{split}
\end{equation}
where $\phi_3(s;x,y,t)=\theta(s)t+s(x-y)$ with $\theta(s)=-(4s^2-2s^2+\frac{1}{4})$ and

 $p_3(s,x,t)=e^{-(s^2-\tfrac{1}{2})^{\tfrac{1}{2}} x}$.

\begin{lem}\label{van2}
  Let $$I(s;\omega,t) \equiv\int_{\tfrac{1}{\sqrt{2}}}^s e^{i\xi\omega-i(4\xi^4-2\xi^2+\frac{1}{4})t}d\xi.$$ Then, $$|I(s;\omega,t)|\le ct^{-1/4},$$ where $c>0$ is independent of $\omega\in\mathbb{R},t>0$, and $s>1/\sqrt{2}$.
\end{lem}

\begin{proof}
The proof of the above lemma is similar to the proof of Lemma \ref{van}, therefore we only mention a few  details here. Let us note that in the above integral we set $$\phi_{t,\omega}(\xi)=4\xi^4-2\xi^2+\frac{1}{4}+\xi\frac{\omega}{t}.$$ Then,
$$\phi_{t,\omega}^{(4)}(\xi)=96\ge 1, \phi_{t,\omega}^{(3)}(\xi)=96\xi, \phi_{t,\omega}^{(2)}(\xi)=48\xi^2-4, \phi_{t,\omega}^{(1)}(\xi)=16\xi^3-4\xi+\frac{\omega}{t}.$$  Again, we set $\delta=t^{-1/4}$.  We can assume without loss of generality that $s>\delta/96+1/\sqrt{2}$ because otherwise the lemma is immediate.

If $\delta/96<1/\sqrt{2}$, then $|\phi_{t,\omega}^{(3)}(\xi)|\ge \delta$ for all $\xi\in [1/\sqrt{2},s].$ Also $$|\phi_{t,\omega}^{(2)}(\xi)|\ge 20\ge \frac{10}{48^2}\delta^2$$ and Van der Corput arguments apply.

On the other hand, if $\delta/96\ge 1/\sqrt{2}$, then $|\phi_{t,\omega}^{(3)}(\xi)|< \delta$ for $\xi\in [1/\sqrt{2},\delta/96)$ and $|\phi_{t,\omega}^{(3)}(\xi)|\ge \delta$  for $\xi\in [\delta/96,s].$ Also, $$\sqrt{2}\delta/48\ge 2\Rightarrow 2\delta^2/48^2\ge 4\Rightarrow -4\le -2\delta^2/48^2.$$ Therefore, $$|\phi_{t,\omega}^{(2)}(\xi)|=48\xi^2-4\ge \delta^2/192-2\delta^2/48^2=3\delta^2/48$$ for $\xi\in [\delta/96,s]$. Hence, we can split the given integral over two regions as $[1/\sqrt{2},\delta/96)\cup [\delta/96,s]$ and finally use Van der Corput arguments in the second interval.
\end{proof}

Now, we have the following result by combining the above lemma and the behavior of exponentially decaying term.

\begin{lem}\label{Lem3.4}
	The kernel $K_3$ defined in $\eqref{DefK3}$ satisfies
	\begin{equation}
	\abs{K_3(y;x,t)}\leq ct^{-1/4},
	\end{equation}
where $t>0$, $x\in\mathbb{R}_{+}$, and $y\in\mathbb{R}$.
\end{lem}
\begin{proof}
	First set $\Phi(s;\omega,t)=I(s;\omega,t)$. Then as in the proof of Lemma \ref{k1lem}, integrating by parts and using $$\lim_{s\rightarrow \infty}\abs{\Phi(s;\omega,t)p_3(s;x)}=0,$$ we obtain
	\begin{equation}
	\abs{K_3(y;x,t)}\leq \int_{\tfrac{1}{\sqrt{2}}}^\infty\abs{\Phi(s;\omega,t)}\abs{\frac{d}{ds}p_3(s;x,t)}ds.
	\end{equation}
We change variables by setting $m^2=s^2-1/2$, $m>0.$ Then the result follows by Lemma \ref{van2} with $\omega=x-y$ and the following uniform estimate
\begin{equation}\label{Esp3}
\begin{split}
\int_{\tfrac{1}{\sqrt{2}}}^\infty\abs{\frac{d}{ds}p_3(s;x,t)}ds=\int_{0}^{\infty}xe^{-mx}dm<1, x>0.
\end{split}
\end{equation}
\end{proof}
The above lemma gives
\begin{equation}\label{Wb3Linfty}
\norm{W_b^3[g_0,g_1]}_{L_x^\infty(\mathbb{R}_+)}\lesssim t^{-\tfrac{1}{4}}\|\Psi_3\|_{L^1}.
\end{equation}
On the other hand, using $\eqref{Wb3rewrite}$ we have
\begin{equation}\label{wb3l2}
\begin{split}
&\left|W_b^3[g_0,g_1](x,t)\right|^2_{L_x^2(\mathbb{R}_{+})}=\\
&=\frac{1}{(2\pi)^2}\int_{0}^{\infty}\left|\int_{\tfrac{1}{\sqrt{2}}}^\infty e^{isx-(s^2-\tfrac{1}{2})^{\tfrac{1}{2}} x-i(2s^2-\tfrac{1}{2})^2t}G(s+i(s^2-\tfrac{1}{2})^{\tfrac{1}{2}},T)(1+\tfrac{is}{(s^2-\tfrac{1}{2})^{\tfrac{1}{2}}})ds\right|^2dx.\\
&\lesssim \int_0^\infty\left(\int_{\tfrac{1}{\sqrt{2}}}^\infty e^{-(s^2-\tfrac{1}{2})^{\tfrac{1}{2}}x}|G(s+i(s^2-\tfrac{1}{2})^{\tfrac{1}{2}},T)|\left|\tfrac{2s^2-\tfrac{1}{2}}{s^2-\tfrac{1}{2}}\right|^{\tfrac{1}{2}} ds\right)^2dx
\end{split}
\end{equation}
After change of variables and using the boundedness of Laplace transform, we have
\begin{equation}
\begin{split}
&\left|W_b^3[g_0,g_1](x,t)\right|^2_{L_x^2(\mathbb{R}_{+})}\\
&\lesssim \int_0^\infty\left(\int_{0}^\infty e^{-mx}|G((m^2+\tfrac{1}{2})^{\tfrac{1}{2}}+im,T)|\left|\tfrac{(2m^2+\tfrac{1}{2})}{(m^2)}\right|^{\tfrac{1}{2}}dm\right)^2dx\\
&\lesssim \int_{-\infty}^\infty |\widehat{\Psi}_3(s)|^2ds=\|\Psi_3\|_{L^2}^2.\label{Esw3}
\end{split}
\end{equation}

Interpolating between $\eqref{Esw3}$ and $\eqref{Wb3Linfty}$, we obtain
\begin{equation}\label{wb3inter}
\left\|W_b^3[g_0,g_1]\right\|_{L_x^r(\mathbb{R}_+)} \lesssim t^{-(\frac{1}{4}-\frac{1}{2r})}\|\Psi_3\|_{L^{r'}}.
\end{equation}for $r\in [2,\infty]$.

The case $\ell=4$ is similar to that of $W_b^3$.  We have

\begin{equation}\label{Wb4}
\begin{split}
&W_b^4[g_0,g_1](x,t)=\int_{\gamma_4}E(k;x,t)G(k;T)dk\\
&=-\frac{1}{2\pi}\int_{\tfrac{1}{\sqrt{2}}}^\infty e^{-isx-(s^2-\tfrac{1}{2})^{\tfrac{1}{2}} x-i(2s^2-\tfrac{1}{2})^2t}G(-s+i(s^2-\tfrac{1}{2})^{\tfrac{1}{2}},T)(-1+\tfrac{is}{(s^2-\tfrac{1}{2})^{\tfrac{1}{2}}})ds.
\end{split}
\end{equation}
Let $\Psi_4$ be the inverse Fourier transform of
\begin{equation}\label{FTphi4}
\hat{\Psi}_4(s)=\left\{
\begin{array}{ll}
G(-s+i(s^2-\tfrac{1}{2})^{\tfrac{1}{2}},T)\Bigg(-1+\tfrac{is}{(s^2-\tfrac{1}{2})^{\tfrac{1}{2}}}\Bigg), & s\ge \frac{1}{\sqrt{2}}\\
0, & s< \frac{1}{\sqrt{2}}.
\end{array}\right.
\end{equation}

Then, $W_b^4$ takes the following form:
\begin{equation}\label{Wb4rewrite}
\begin{split}
W_b^4[g_0,g_1](x,t)&=-\frac{1}{2\pi}\int_{-\infty}^\infty\Psi_4(y)\int_{\tfrac{1}{\sqrt{2}}}^\infty e^{-isx-(s^2-\tfrac{1}{2})^{\tfrac{1}{2}} x-i(2s^2-\tfrac{1}{2})^2t-isy}dsdy\\
&=:\int_{-\infty}^{\infty}K_4(y;x,t)\Psi_4(y)dy.
\end{split}
\end{equation}

where $K_4(x,y,t)$ is the kernel of $W_b^4$. We can deduce that $K_4$ decays as ${t}^{-1/4}$ by arguing as in the case of $W_b^3$ because we can write
\begin{equation}\label{DefK4}
\begin{split}
K_4(y;x,t)&=\int_{1/\sqrt{2}}^\infty e^{-is(x+y)-i(4s^2-2s^2+\frac{1}{4})t}e^{-(s^2-\tfrac{1}{2})^{\tfrac{1}{2}} x}ds\\
&=\int_{1/\sqrt{2}} e^{i\phi_4(s;x,y,t)}p_4(s;x)ds
\end{split}
\end{equation}
where $\phi_4(s;x,y,t)=\theta(s)t-s(x+y)$ with $\theta(s)=-(4s^2-2s^2+\frac{1}{4})$ and

$p(s;x)=e^{-(s^2-\tfrac{1}{2})^{\tfrac{1}{2}} x}$.

Therefore, we have
	\begin{equation}
	\abs{K_4(y;x,t)}\leq ct^{-1/4},
	\end{equation}
	where $t\neq 0$, $x\in\mathbb{R}_{+}$, and $y\in\mathbb{R}$. This implies
\begin{equation}\label{Wb4Linfty}
\left\|W_b^4[g_0,g_1]\right\|_{L_x^\infty(\mathbb{R}_+)}\lesssim t^{-\tfrac{1}{4}}\|\Psi_4\|_{L^1}.
\end{equation}
Again, from the boundedness of the Laplace transform we have
\begin{equation}\label{Esw4}
\left\|W_b^4[g_0,g_1]\right\|_{L_x^2(\mathbb{R}_+)}\lesssim \|\Psi_4\|_{L^2}.
\end{equation}
Interpolating between $\eqref{Esw4}$ and $\eqref{Wb4Linfty}$, we obtain
\begin{equation}\label{wb4inter}
\left\|W_b^4[g_0,g_1]\right\|_{L_x^r(\mathbb{R}_+)} \lesssim t^{-(\frac{1}{4}-\frac{1}{2r})}\|\Psi_4\|_{L^{r'}}.
\end{equation}for $r\in [2,\infty]$.

\section{Towards the general case}

In what follows we present a general form of the integral boundary operator \eqref{bdintop2}, namely the \texttt{BdIntOp} for the problem \eqref{maineq}-\eqref{bdry}, with the linear operator $P$ be defined by the general form
$$P=-i(\alpha\partial_x^4+\beta\partial_x^2), \qquad  \beta\in\mathbb{R}, \ \alpha\in\mathbb{R}, \ \ \alpha\ne 0.$$
We note that the case $\alpha=\beta=1$ was discussed earlier in this paper and the case $\alpha=1, \ \beta=0$ was analysed in \cite{OY19}.

For the general problem the global relation \eqref{IntythatktRe} takes the form
\begin{multline}\label{IntythatktRe-g}
e^{w(k)t}\hat{y}(k,t)=-i\alpha\tilde{g}_3(w(k),t)+\alpha k\tilde{g}_2(w(k),t)\\
-i(\beta-\alpha k^2)\tilde{g}_1(w(k),t)+k(\beta-\alpha k^2)\tilde{g}_0(w(k),t), \quad \Im\,k\le 0,
\end{multline}
with
\begin{equation}\label{def:w-g}
w(k)=-i(\alpha k^4-\beta k^2).
\end{equation}
Furthermore, the integral representation of the solution takes the form \eqref{qxt} with $\tilde{g}$ being defined as
$$\tilde{g}=i\alpha\tilde{g}_3- \alpha k\tilde{g}_2+(\beta-\alpha k^2)(i\tilde{g}_1-k\tilde{g}_0).$$

Following the same arguments as in section \ref{seccon} we are able to derive the general boundary integral operator
\begin{equation}\label{bdintop-g}
  W_b^g[g_0,g_1](x,t) =\alpha \int_{\partial D^+}E(k;x,t)G(k;t)dk,
\end{equation}
where $E$ is defined in \eqref{def:E}, $G$ is defined in \eqref{def:G}. Therein $w(k)$ given by \eqref{def:w-g} and $\nu(k)$ is defined follows \begin{equation}\label{nukmod-g}
\nu(k)=\sqrt{\frac{\beta}{\alpha}-k^2}^*,
\end{equation}
where we set $\displaystyle \sqrt{z}^*:=|z|^{\frac{1}{2}}e^{i\frac{\arg z}{2}}$ and  for some fixed and sufficiently small $\epsilon>0$ we choose:
\begin{itemize}
\item $\arg z\in [-\pi +\epsilon,\pi+\epsilon)$, for $\alpha>0,\ \beta>0$. Then $\nu(k)$ is analytic on $\overline{D^+}\setminus \left\{-\sqrt{\frac{\beta}{\alpha}}\right\}$ (See Figure \ref{Branchcut-g1}).
\begin{figure}[!htb]
	\centering
	\includegraphics[width=0.7\textwidth]{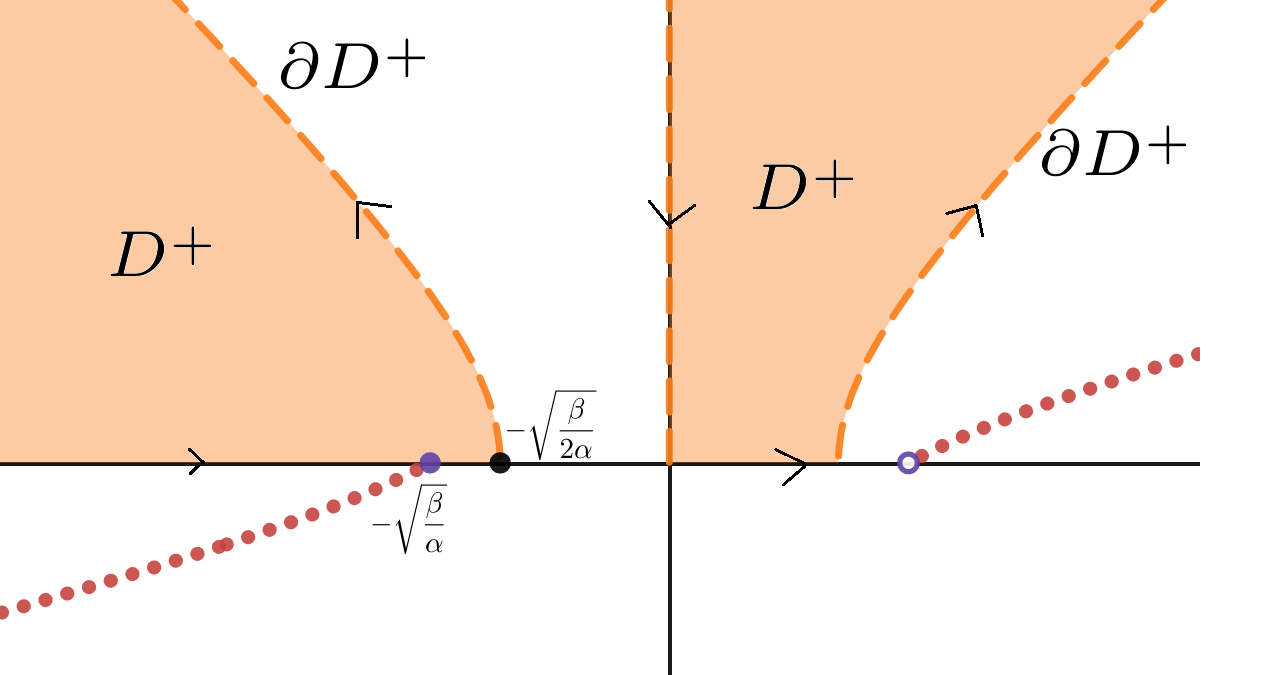}
	\caption{Red path denotes the branch cut of $\nu(k)=\sqrt{\frac{\beta}{\alpha}-k^2}^*$ for $\alpha>0,\ \beta>0$.}\label{Branchcut-g1}
\end{figure}
\item $\arg z\in [\epsilon,2\pi+\epsilon)$, for $\alpha>0,\ \beta<0$.  Then $\nu(k)$ is analytic on $\overline{D^+}\setminus \left\{ i\sqrt{-\frac{\beta}{\alpha}}\right\}$ (See Figure \ref{Branchcut-g2}).
\begin{figure}[!htb]
	\centering
	\includegraphics[width=0.5\textwidth]{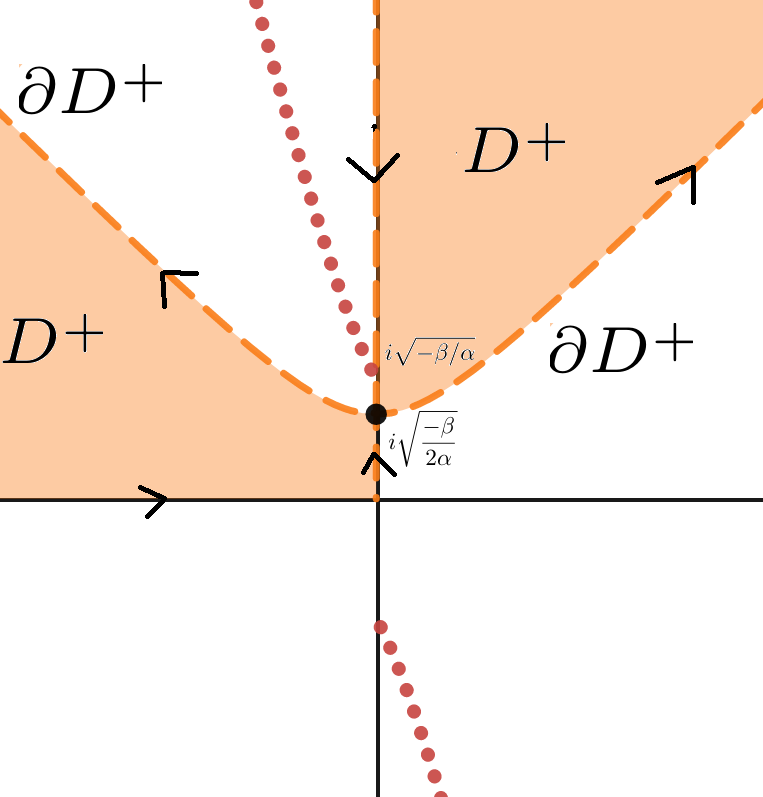}
	\caption{Red path denotes the branch cut of $\nu(k)=\sqrt{\frac{\beta}{\alpha}-k^2}^*$ for $\alpha>0,\ \beta<0$.}\label{Branchcut-g2}
\end{figure}
\item $\arg z\in [-\pi- \epsilon, \pi-\epsilon)$, for $\alpha<0,\ \beta<0$. Then $\nu(k)$ is analytic on $\overline{D^+}\setminus \left\{\sqrt{\frac{\beta}{\alpha}}\right\}$ (See Figure \ref{Branchcut-g3}).
\begin{figure}[!htb]
	\centering
	\includegraphics[width=0.5\textwidth]{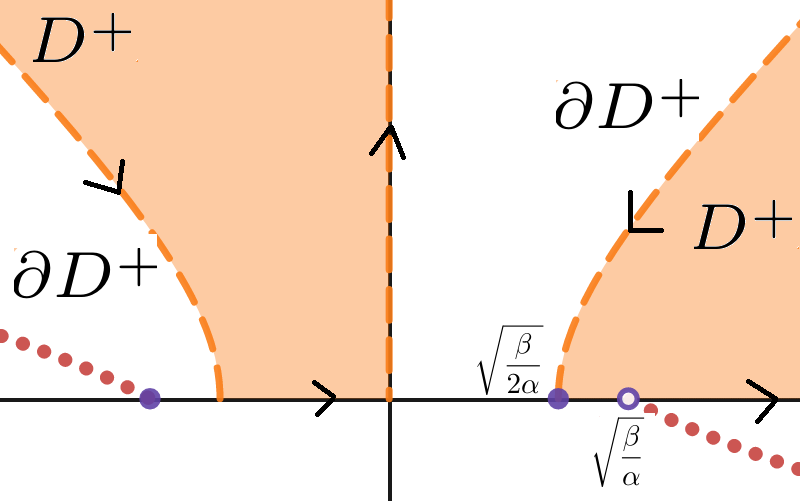}
	\caption{Red path denotes the branch cut of $\nu(k)=\sqrt{\frac{\beta}{\alpha}-k^2}^*$ for $\alpha<0,\ \beta<0$.}\label{Branchcut-g3}
\end{figure}
\item $\arg z\in [-\epsilon,2\pi-\epsilon)$, for $\alpha<0,\ \beta>0$. Then $\nu(k)$ is analytic on $\overline{D^+}\setminus \left\{i\sqrt{-\frac{\beta}{\alpha}}\right\}$ (See Figure \ref{Branchcut-g4}).
\begin{figure}[!htb]
	\centering
	\includegraphics[width=0.5\textwidth]{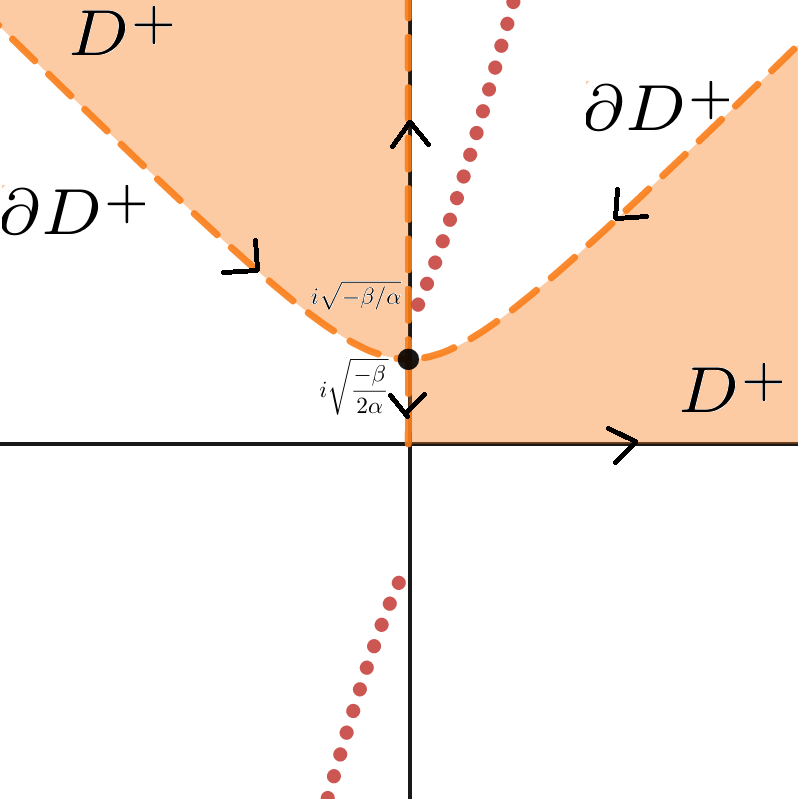}
	\caption{Red path denotes the branch cut of $\nu(k)=\sqrt{\frac{\beta}{\alpha}-k^2}^*$ for $\alpha<0,\ \beta>0$.}\label{Branchcut-g4}
\end{figure}
\end{itemize}

For the case that $\alpha=1$ and $\beta=0$ the \texttt{BdIntOp} \eqref{bdintop-g} simplifies to \eqref{ysol}, since \eqref{nukmod-g} takes the form
$$\nu(k)=\begin{cases} ik, & k \in  D_1^+=\left\{ k:\arg k\in \left(\frac{\pi}{4},\frac{\pi}{2} \right) \right\} \\ -ik, & k \in D_2^+ =\left\{ k: \arg k\in  \left(\frac{3\pi}{4},\pi\right)\right\} \end{cases}.$$


\appendix
\label{appendix}
\section{Proof of Lemma \ref{van}}
\label{ProofofLemma31}
\begin{proof}We can write
	$$I(s;y,t) \equiv \int_{0}^{s}e^{i(\xi^4+\xi^2)t-i\xi y}d\xi=\int_{0}^{s}e^{it\left(\xi^4+\epsilon \xi^2-\frac{\xi y}{t}\right)}d\xi.$$
	Set $\phi_{t,y}(\xi)\equiv \xi^4+\xi^2-\frac{\xi y}{t}.$ Then, $\phi_{t,y}^{(4)}(\xi)=24\ge 1$. We can use the steps of the proof of Van der Corput lemma and prove that   $$|I(s;y,t)|\lesssim t^{-1/4}, t>0,$$ where the constant of the inequality is independent of $\phi_{t,y}$, $y$, $t$, and $s$.  For completeness we give the details because we refer to the content of this lemma for other oscillatory integrals later.
	
	Indeed, we first set $\delta\equiv t^{-1/4}$. Then, we have  $\phi_{t,y}^{(3)}(\xi)=24\xi< \delta$ if $\xi\in [0,\delta/24).$ Therefore, we can write $$I(s;y,t)=\int_{0}^{\delta/24}e^{i(\xi^4+\xi^2)t-i\xi y}d\xi+\int_{\delta/24}^{s}e^{i(\xi^4+ \xi^2)t-i\xi y}d\xi\equiv A+B.$$ Clearly, $|A|\le \frac{\delta}{24}= \frac{t^{-1/4}}{24}.$ We are assuming without loss of generality that $s>\delta/24$, otherwise the result of the lemma is immediate.
	
	Now, we will estimate $B.$ First observe that $\phi_{t,y}''(\xi)=12\xi^2+2\ge \frac{\delta^2}{48}+2$ on $[\delta/24,s].$ In particular, $\phi_{t,y}'$ is  monotone on $[\delta/24,s].$ Now, we define $m_{t,y}$ be such that $$|\phi'(m_{t,y})|= \inf_{\xi\in [\delta/24,s]}|\phi_{t,y}'(\xi)|.$$ There is only one such point in $[\delta/24,s]$ due to monotonicity of $\phi_{t,y}'$. Note that $\phi_{t,y}'(\xi)=4\xi^3+2\xi-\frac{y}{t}$. Only three cases are possible:
	\begin{itemize}
		\item[(i)] The first case is $m_{t,y}\in [\delta/24,s]$ and $\phi_{t,y}'(m_{t,y})=0$. If $\xi\notin (m_{t,y}-\delta,m_{t,y}+\delta)$, we have $$|\phi_{t,y}'(\xi)|=\left|\int_{m_{t,y}}^{\xi}\phi_{t,y}''(\xi)d\xi\right|\ge \left(\frac{\delta^2}{48}+2\right)|\xi-m_{t,y}|\ge \frac{\delta^3}{48}+2\delta.$$  We write, $$B=\int_{\delta/24}^{s}\cdot=\int_{\delta/24}^{m_{t,y}-\delta}\cdot+\int_{m_{t,y}-\delta}^{m_{t,y}+\delta}\cdot+\int_{m_{t,y}+\delta}^{s}\cdot\equiv \sum_{i=1}^{3}B_i.$$ Clearly, $|B_2|\le 2\delta=2t^{-1/4}.$ Let us estimate $B_1$. We integrate by parts, use the monotonicity of $\phi_{t,y},$ Fundamental Theorem of Calculus and obtain
		\begin{equation}
			\begin{split}
				|B_1|&=\left|\int_{\delta/24}^{m_{t,y}-\delta}e^{it\phi_{t,y}(\xi)}d\xi\right|\\
				&\le\left|\frac{e^{it\phi_{t,y}(\xi)}}{it\phi_{t,y}'(\xi)}\right|_{\delta/24}^{m_{t,y}}+\int_{\delta/24}^{m_{t,y}-\delta}\left|\frac{d}{d\xi}\left(\frac{1}{it\phi_{t,y}'(\xi)}\right)\right|d\xi\\
				&\le \frac{2}{t}\left(\frac{\delta^3}{48}+2\delta\right)^{-1}+\frac{1}{t}\left|\int_{\delta/24}^{m_{t,y}-\delta}\frac{d}{d\xi}\left(\frac{1}{\phi_{t,y}'(\xi)}\right)d\xi\right|\\
				&\le \frac{4}{t}\left(\frac{\delta^3}{48}+2\delta\right)^{-1}\le 192t^{-1/4}.
			\end{split}
		\end{equation}
		$|B_3|$ is estimated in the same manner and we can find the same bound for it.
		\item[(ii)] Consider the case $\phi_{t,y}'(m_{t,y})\neq 0$ and $m_{t,y}=\delta/24$. In this case, we decompose as
		$$B=\int_{\delta/24}^{s}\cdot=\int_{\delta/24}^{m_{t,y}+\delta}\cdot+\int_{m_{t,y}+\delta}^{s}\cdot\equiv C_1+C_2,$$ where $|C_1|\le \delta$ and $|C_2|\le 192t^{-1/4}$ by the same arguments in (i).
		\item[(iii)] Consider the case $\phi_{t,y}'(m_{t,y})\neq 0$ and $m_{t,y}=s$. In this case, we decompose as
		$$B=\int_{\delta/24}^{s}\cdot=\int_{\delta/24}^{m_{t,y}-\delta}\cdot+\int_{m_{t,y}-\delta}^{s}\cdot\equiv D_1+D_2,$$ where $|D_1|\le 192t^{-1/4}$ and $|D_2|\le \delta$ by the same arguments in (i) and (ii).
		By the three cases above, the lemma follows.
	\end{itemize}
\end{proof}
 \bibliographystyle{plain}
\bibliography{myrefs}
\end{document}